\documentclass[11pt]{amsart}

\usepackage{amsthm}
\usepackage{amssymb}
\usepackage{graphicx}									 
\usepackage{stmaryrd}                  
\usepackage{footmisc}                  
\usepackage{paralist}
\usepackage{wrapfig}
\usepackage[draft]{pdfcomment}
\usepackage{bbm}
\usepackage[T1]{fontenc}
\usepackage[utf8]{inputenc}
\usepackage[arrow, matrix, curve]{xy}
\usepackage{color}

\usepackage{tikz-cd}

\usepackage{enumitem}

\newtheorem{thm}{Theorem}[section]

\newtheorem{lem}[thm]{Lemma}

\newtheorem{cor}[thm]{Corollary}

\theoremstyle{definition}

\theoremstyle{remark}

\newcommand{\R}{\mathbb{R}}

\newcommand{\N}{\mathbb{N}}

\newcommand{\To}{\rightarrow}

\newcommand{\Orb}{\mathcal{O}}

\newcommand{\Or}{\mathrm{O}}

\title[Orbifolds and manifold quotients with upper curvature bounds]{Orbifolds and Manifold quotients \\ with upper curvature bounds}

\author[C. Lange]{Christian Lange}
\address{Ludwig-Maximilians-Universit\"at M\"unchen, Mathematisches Institut\newline\indent Theresienst. 39, 80333 München, Germany}
\email{lange@math.lmu.de, clange.math@gmail.com}

\begin{document}

\maketitle
\begin{abstract}
We characterize Riemannian orbifolds with an upper curvature bound in the Alexandrov sense as reflectofolds, i.e. Riemannian orbifolds all of whose local groups are generated by reflections, with the same upper bound on the sectional curvature. Combined with a result by Lytchak--Thorbergsson this implies that a quotient of a Riemannian manifold by a closed group of isometries has locally bounded curvature (from above) in the Alexandrov sense if and only if it is a reflectofold.
\end{abstract}

\section{Introduction}

Let $M$ be a Riemannian manifold and let $G$ be a closed group of isometries of $M$. Then the quotient space $M/G$ is a metric space whose metric properties are often related to properties of the action in an interesting way,  see e.g. \cite{GLLM22,LT10}. Usually, this quotient is not a Riemannian manifold, but an Alexandrov space with curvature locally bounded from below. Nevertheless, it is stratified by Riemannian manifolds and the so-called principle stratum is open and dense \cite{AB15}. Lytchak and Thorbergsson have shown that the sectional curvature of this principle stratum is bounded from above in the neighborhood of a point if and only if this neighborhood in $M/G$ is a Riemannian orbifold, i.e. a metric space which is locally isometric to the quotient of a Riemannian manifold by an isometric action of a finite group \cite[Theorem~1.1]{LT10}. However, the curvature of such a quotient is in general still locally unbounded from above in the Alexandrov sense. For instance, the quotient of $\R^2$ by a finite cyclic group of rotations around the origin is isometric to the Euclidean cone over a circle of radius $2\pi/k$ for some $k>1$ and exhibits infinite positive curvature at the tip of the cone. Infinitesimally, the only exceptions of this phenomenon one can think of are quotients of $\R^n$ by finite reflection groups \cite{Hu90}. In this case the quotient is isometric to a Weyl chamber of the corresponding reflection group and thus flat. Globally, these examples correspond to so-called reflectofolds, i.e. Riemannian orbifolds all of whose local groups are reflection groups \cite{Da11}. In particular, reflectofolds are Riemannian manifolds in their interior. In fact, metric spaces with two-sided curvature bounds are always Riemannian manifolds (perhaps of low regularity) in their interior \cite{BN93}, cf. \cite[Theorem~10.10.13]{BBI01}. Here we confirm that for the whole quotient space no other examples than reflectofolds locally have a two-sided curvature bound.

\begin{thm}\label{thm:main_thm} The curvature of a Riemannian orbifold is locally bounded from above in the Alexandrov sense if and only if it is a reflectofold. In this case, locally, the curvature is bounded from above by $k$ in the Alexandrov sense if and only if the sectional curvature is  bounded from above by $k$.
\end{thm}

In the manifold case this statement is due to Alexandrov \cite{Al51} based on earlier work by Cartan \cite{Ca28}.
Our proof of Theorem \ref{thm:main_thm} works by induction on the dimension and uses the fact that tangent spaces of metric spaces with curvature bounded from above are CAT($0$), see Theorem \ref{thm:tangent_cat}. The main step is to prove Lemma \ref{lem:local_cat_char} which characterizes CAT($0$) spaces among quotient spaces $\R^n/\Gamma$ for finite subgroups $\Gamma <\Or(n)$. Namely, such a quotient is CAT($0$) if and only if $\Gamma$ is generated by reflections, in which case the quotient is isometric to a Weyl chamber of $\Gamma$ if $\Gamma$ is nontrivial. The strategy of the proof is related to the one in \cite[Section~4]{La16} and \cite[Lemma~4.17]{La19}.

Based on the result by Lytchak--Thorbergsson \cite[Theorem~1.1]{LT10} we state the following corollary.

\begin{cor} \label{thm:gen_m_g_cat} Let $M$ be a Riemannian manifold and let $G$ be a closed group of isometeries of $M$. Let $p \in M$ be a point with isotropy $G_p$. Set $\bar p=Gp \in M/G$. Then the following are equivalent:
\begin{enumerate}
\item The curvature of $M/G$ is bounded from above in a neighborhood of $\bar p$ in the Alexandrov sense.
\item A neighborhood of $\bar p$ in $M/G$ is a reflectofold.
\item The action of $G_p$ on $T_pM$ is polar and and the corresponding polar group $\Pi$ is generated by reflections.
\end{enumerate}
\end{cor}

Here the action of $G_p$ on $T_pM$ is called polar if there exists a subspace $\Sigma \subseteq T_pM$ that meets all orbits of the $G_p$-action and meets them always orthogonally. The polar group $\Pi$ is defined to be the quotient of the subgroup that leaves $\Sigma$ invariant modulo the subgroup that fixes $\Sigma$ pointwise. It acts naturally on $\Sigma$. In particular, condition $(iii)$ is satisfied, if $G_p$ is connected \cite[Proposition~1.4]{GZ12} and if $G_p$ is Coxeter polar like the isotropy representations of a symmetric space. For more details on polar action we refer to the survey \cite{GZ12}.
\section{Quotients with upper curvature bounds}\label{sec:proof}

\subsection{Spaces with upper curvature bounds}

We say that a metric space $X$ is $D$-geodesic for some $D>0$ if all points of distance less than $D_k$ are connected by a geodesic. Let $D_k$ be the diameter of the complete model plane of constant curvature $k$. We say that $X$ is CAT($k$) if it is $D_k$-geodesic and all geodesic triangles in $X$ of perimeter less than $2D_k$ satisfy the CAT($k$) comparison condition, see \cite[II.1]{BH99}. We say that $X$ has curvature bounded from above by $k$ in the Alexandrov sense if it is locally CAT($k$). This definition is motivated by the following result of Alexandrov. A proof can also be found in \cite[Theorem~1A.6]{BH99}. 

\begin{thm}[Alexandrov, \cite{Al51}]\label{thm:alexandrov_char} A Riemannian manifold has curvature bounded from above by $k$ in the Alexandrov sense if and only if its sectional curvature is bounded from above by $k$.
\end{thm}

For a metric space $X$ with curvature bounded from above in the Alexandrov sense we denote the (completion of the) space of directions at a point $p\in X$ by $\Sigma_p X$ and the tangent cone of $X$ at $p$, which is isometric to the Euclidean cone over the space of directions $\Sigma_p X$, by $T_pX$, cf. \cite[II 3]{BH99}. A proof of the following statement, first outlined by Kleiner and Leeb \cite{KL97}, can be found in \cite[II 3.19]{BH99}.

\begin{thm}[Nikolaev, \cite{N95}]\label{thm:tangent_cat} If a metric space has curvature bounded from above in the Alexandrov sense, then all spaces of directions are  $\mathrm{CAT}(1)$ and all tangent spaces are  $\mathrm{CAT}(0)$.
\end{thm}

\subsection{Riemannian orbifolds}

A Riemannian $n$-orbifold $\Orb$ can be defined as a length space which is locally isometric to the quotient of a Riemannian $n$-manifold by an isometric action of a finite group \cite{La20}. The isotropy group of the preimage of a point $p\in \Orb$ in such a manifold chart under the finite group action is uniquely defined up to conjugation in $\Or(n)$, and it is called the local group of $\Orb$ at $p$. The set of points with trivial local group is called the regular part of $\Orb$. From the local model it is easy to see that the regular part is open and dense. Hence, if the curvature of a Riemannian orbifold is bounded from above by $k$ in the Alexandrov sense, then the sectional curvature of its regular part must satisify the same curvature bound by Theorem \ref{thm:alexandrov_char}. Moreover, since the regular part is dense, the sectional curvature bound must be satisfied everywhere.

\subsection{Proof of Theorem \ref{thm:main_thm} and Corollary \ref{thm:gen_m_g_cat}}

The space of directions of a Riemannian orbifold $\Orb$ at a point $p\in \Orb$ with local group $\Gamma_p$ is isometric to $S^{n-1}/\Gamma_p$, where $S^{n-1}$ denotes the unit sphere in $\R^n$, and the tangent cone of $\Orb$ at $p$ is isometric to $\R^n/\Gamma_p$. Therefore, the statement in Theorem \ref{thm:main_thm} that local groups are generated by reflections if the curvature is bounded from above in the Alexandrov sense is a consequence of the following lemma.

\begin{lem}\label{lem:local_cat_char}
Let $\Gamma < \Or (n)$ be a finite subgroup such that $\R^n/\Gamma$ is $\mathrm{CAT}(0)$. Then $\Gamma$ is generated by reflections.
\end{lem}
\begin{proof} We prove the claim by induction on $n$. For $n=1$ the claim is obvious.

Assume that the claim holds for some $n \in \N$ and let $\Gamma < \Or (n+1)$ be a finite subgroup such that $\R^{n+1}/\Gamma$ is CAT(0). Then $\Sigma_{\overline{0}}(\R^{n+1}/\Gamma)=S^n/\Gamma$ is CAT(1) by Theorem \ref{thm:tangent_cat}. Hence, for any $v\in S^n$ we have that $T_{\overline{v}}(S^{n}/\Gamma)=T_v S^n/\Gamma_v=\R^n/\Gamma_v$ is CAT(0) and so for any $v\in S^n$ the isotropy group $\Gamma_v$ is generated by reflections by our induction assumption.

We set $\Gamma' = \left\langle \Gamma_v \mid v \in S^n \right\rangle$. By construction $\Gamma'$ is a normal subgroup of $\Gamma$ generated by reflections. The induced action of $\Gamma/\Gamma'$ on $S^n/\Gamma'$ is isometric and free, see \cite[Lemma~4.16]{La19}.

We first suppose that $\Gamma'$ is nontrivial. In this case $\R^{n+1}/\Gamma'$ is isometric to a Weyl chamber \cite{Hu90} and so the quotient $S^n/\Gamma'$  is contractible. This implies $\Gamma=\Gamma'$ since a nontrivial finite group cannot act freely on a finite dimensional complex (although it can act without fixed points \cite{FR59}). Alternatively, we can also argue geometrically as follows. The group $\Gamma$ leaves the subspace 
\[V=\mathrm{Fix}(\Gamma')=\{v\in \R^n \mid gv = v \text{ for all } g\in \Gamma'\}\cong \R^k\]
and its orthogonal complement invariant. Since $\Gamma'$ is nontrivial we have that $k>0$. Then $\Delta=S^{n-k}/\Gamma'$ is a (strictly convex) spherical simplex \cite{Hu90}. The join splitting $S^n=S^{k-1} * S^{n-k}$ induces a splitting $S^n/\Gamma'=S^{k-1} *\Delta$ of its $\Gamma'$-quotient whose subspaces $S^{k-1}$ and $\Delta$ are invariant under the $\Gamma/\Gamma'$ action. Since the barycenter of $\Delta$ is fixed by $\Gamma/\Gamma'$, we again conclude that $\Gamma=\Gamma'$.

Hence, we can assume that $\Gamma'$ is trivial. In this case $\Gamma$ acts freely on $S^n$. Therefore, for any nontrivial $g\in \Gamma$ there exists a periodic geodesic on $S^n$ which together with its orientiation is preserved by $g$. Suppose that there is such a nontrivial $g$. Then a corresponding invariant geodesic projects onto a periodic geodesic of $S^n/\Gamma$ of length $<2\pi$, again because the action of $\Gamma$ on $S^n$ is free. This contradicts the fact that $S^n/\Gamma$ is CAT($1$), see \cite[Proposition~2.2.7]{AKP19}, and so $\Gamma$ has to be trivial as well in this case. The claim follows.
\end{proof}

This completes the proof of the if direction of Theorem \ref{thm:main_thm}. The proof of the only if direction is based on the fact that the quotient $\R^n/\Gamma$ of $\R^n$ by a reflection group $\Gamma<\Or(n)$ is isometric to a Weyl chamber in $\R^n$ \cite{Hu90}. If a finite group $\Gamma_p$ acts isometrically on a Riemannian manifold $M$ fixing a point $p\in M$ such that the induced action of $\Gamma_p$ on $T_pM$ is generated by reflections, then a small $r$-neighborhood of the image of $p$ in $M/\Gamma_p$ is isometric to the image under the exponential map $\exp_p: T_pM \To M$ of the intersection of a corresponding Weyl chamber with a small $r$-neighborhood of $0\in T_pM$. The claim now follows from the observation that for sufficiently small $r>0$ this image is a convex subset of $M$, which is an easy exercise in Riemannian geometry.

\begin{proof}[Proof of Corollary \ref{thm:gen_m_g_cat}]
To prove Corollary \ref{thm:gen_m_g_cat} we first observe that an upper curvature bound in the Alexandrov sense for a neighborhood $U$ in $M/G$ implies an upper bound on the sectional curvature of the intersection of $U$ with the principle stratum of $M/G$ by Theorem \ref{thm:alexandrov_char}. In this case $U$ is a Riemannian orbifold by \cite[Theorem~1.1]{LT10} and hence a reflectofold by Theorem~\ref{thm:main_thm}. The reverse implication also follows from Theorem~\ref{thm:main_thm}. 

The other equivalence follows from \cite[Theorem~1.1]{LT10} together with the observation that the quotient $\Sigma/\Pi$ of a section $\Sigma$ of the polar action of $G_p$ on $T_pM$ modulo the polar group $\Pi$ is isometric to the tangent cone of the orbifold $M/G$ at $\overline{p}$
\end{proof}

\emph{Acknowledgements.} The author thanks Claudio Gorodoski for discussions about polar actions and the members of the LMU geometry seminar for useful questions and comments.

\end{document}